\documentclass{article}
\usepackage[english]{babel}
\usepackage{amsmath,amssymb,eufrak,amsthm}
\usepackage{pdfsync}
\begin{document}
\bibliographystyle{plain}

\renewcommand{\sectionmark}[1]{\markboth{#1}{}}

\def\move-in{\parshape=1.75true in 5true in}

\def\CC{\mathbb C}
\def\ZZ{\mathbb Z}
\def\NN{\mathbb N}
\def\RR{\mathbb R^}
\def\PP{\mathbb P^}
\def\CA{\mathcal A}
\def\G{\mathfrak{S}}

\def\sgn{\mathrm{sgn}}
\def\dim{\mathrm{dim}}
\def\codim{\mathrm{codim}}
\def\rk{\mathrm{rk}}
\def\v{\underline{v}}
\def\X{X^{(n,m)}_{(1,d)}}

\def\Def#1{\noindent {\bf Definition #1:}}
\def\Not{\noindent {\bf Notation: }}
\def\Prop#1{\noindent {\bf Proposition #1:}}
\def\Proof{\noindent {\it Proof: }}
\def\Obs{\noindent{\bf Remark: }}
\def\Ex{\noindent{\bf Example: }}
\newenvironment{dem}{\begin{proof}[Proof]}{\end{proof}}
\newtheorem{theorem}{Theorem}[section]
\newtheorem{lemma}[theorem]{Lemma}
\newtheorem{propos}[theorem]{Proposition}
\newtheorem{corol}[theorem]{Corollary}
\newtheorem{defi}[theorem]{Definition}
\newtheorem{conj}[theorem]{Conjecture}
\newtheorem{rem}[theorem]{Remark}
%%%%%%%%%%%%%%%%%%%%

\title{Higher secant varieties of $\PP n \times\PP m$ embedded in bi-degree $(1,d)$}
\author{Alessandra Bernardi\footnote{CIRM--FBK  c/o Universit\`a degli Studi di Trento, via Sommarive 14,
38050 Povo (Trento), Italy. E-mail address: {\tt{bernardi@fbk.eu}}}, Enrico Carlini\footnote{Dipartimento di Matematica
Politecnico di Torino, Corso Duca Degli Abruzzi 24, 10129 Torino, Italy. E-mail address: {\tt carlini@calvino.polito.it}}, Maria Virginia Catalisano\footnote{DIPTEM - Dipartimento di Ingegneria della Produzione,
Termoenergetica e Modelli Matematici, Piazzale Kennedy, pad. D, 16129 Genoa, Italy.
E-mail address: {\tt catalisano@diptem.unige.it}}}

\date{}
\maketitle

%%%%%%%%%%%%%%%%%%%%

\begin{abstract}
Let $X^{(n,m)}_{(1,d)}$ denote the Segre\/-Veronese embedding of
$\PP n \times \PP m$ via the sections of the sheaf
$\mathcal{O}(1,d)$. We study the dimensions of higher secant
varieties of $X^{(n,m)}_{(1,d)}$ and we prove that there is no
defective $s^{th}$ secant variety, except
 possibly for $n$ values of $s$. Moreover when ${m+d \choose d}$ is multiple of $(m+n+1)$,
 the $s^{th}$ secant variety of  $X^{(n,m)}_{(1,d)}$
  has the expected dimension for every $s$.
\end{abstract}

%%%%%%%%%%%%%%%%%%%%
\section*{Introduction}
%%%%%%%%%%%%%%%%%%%%

The $s^{th}$ higher secant variety of a projective variety $X\subset \PP
N$ is defined to be the Zariski closure of the union of the span
of $s$ points of $X$ (see Definition \ref{secant}), we will denote
it with $\sigma_{s}(X)$.

Secant varieties have been intensively studied (see for example
\cite{AH}, \cite{BCS}, \cite{Gr}, \cite{Li}, \cite{St},
\cite{Za}). One of the first problems of interest is the
computation of their dimensions. In fact, there is an expected
dimension for $\sigma_{s}(X)\subset \PP N$, that is, the minimum
between $N$ and $s(\dim X)+s-1$. There are well known
examples where that dimension is not attained, for instance,  the variety of
secant lines to the Veronese surface in $\mathbb{P} ^5$. A
variety $X$  is said to be $(s-1)$-defective if there exists
an integer $s\in \mathbb{N}$ such that the dimension of
$\sigma_{s}(X)$ is less than the expected value. We would like to
notice that  only for Veronese varieties a complete list of all
defective cases is given. This description is obtained using a
result by J. Alexander and A. Hirschowitz \cite{AH} recently
reproposed with a simpler proof in \cite{BO}.

The interest around these varieties has been recently revived from
many different areas of mathematics and applications when $X$ is a
variety parameterizing certain kind of tensors (for example
Electrical Engineering - Antenna Array Processing \cite{ACCF},
\cite{DM} and Telecommunications \cite{Ch}, \cite{dLC} -
Statistics -cumulant tensors, see \cite{McC} -, Data Analysis -
Independent Component Analysis \cite{Co1}, \cite{JS} -; for other
applications see also \cite{Co2}, \cite{CR}, \cite{dLMV},
\cite{SBG}  \cite{GVL}).

One of the main examples is the one of Segre varieties. Segre
varieties parameterize completely decomposable tensors (i.e.
projective classes of tensors in $\PP {}(V_1\otimes \cdots \otimes
V_t)$ that can be written as $v_1\otimes \cdots \otimes v_t$ ,
with $v_i\in V_i$ and $V_i$ vector spaces for $i=1, \ldots , t$).
The $s^{th}$ higher secant varieties of Segre varieties is
therefore the closure of the sets of tensors that can be written
as a linear combination of $s$ completely decomposable tensors.

Segre-Veronese varieties can be described both as  the embedding
of $\PP {n_{1}} \times \cdots \times  \PP {n_{t}}$ with the
sections of the sheaf ${\cal O}(d_{1}, \ldots ,d_{t})$ into $\PP
{N}$, for certain $d_{1}, \ldots , d_{t}\in \NN$, with
$N=\Pi_{i=1}^{t}{n_{i}+d_{i} \choose d}-1$, both as a section of
Segre varieties. Consider the  Segre variety that naturally lives
in $\PP {}(V_1^{\otimes d_1}\otimes \cdots \otimes V_t^{\otimes
d_t})$ with $V_i$ vector spaces of dimensions $n_i+1$ for $i=1,
\ldots , t$, then the Segre-Veronese variety is obtained
intersecting that Segre variety with  the projective subspaces
$\PP {}(S^{d_1}V_1\otimes \cdots \otimes S^{d_t}V_t)$ of
projective classes of partially symmetric tensors
(where $S^{d_i}V_i \subset V_i^{\otimes d_i}$ is the subspace of completely symmetric tensors of $V_i^{\otimes d_i}$).

 These two
different ways of describing Segre-Veronese varieties allow us to
translate problems  about partially symmetric tensors into
problems on forms of multi-degree $(d_{1}, \ldots , d_{t})$ and
viceversa. We will follow the description of Segre-Veronese
variety as the variety parameterizing forms of certain
multi-degree.

In this paper we will describe the $s^{th}$ higher secant varieties of
the embedding of  $\PP n \times \PP m$ into $\PP N$ ($N=(n+1){m+d \choose
d}-1$ ), by the sections of the
sheaf $\mathcal{O}(1,d)$, for almost all $s \in \mathbb N$ (see Theorem \ref{t2}).

The higher secant varieties of the  Segre embedding of $\PP n
\times \PP m$ are well known as they parameterize matrices of
bounded rank (e.g., see \cite{Hr}).

One of the first instance of the study of the two factors
Segre-Veronese varieties is the one of $\PP 1 \times \PP 2$
embedded in bi-degree $(1,3)$ and appears in a paper by London
\cite{London}, for a more recent
approach see \cite{DF} and \cite{CaCh}. A  first generalization
for $\PP 1 \times \PP 2$ embedded in bi-degree $(1,d)$ is treated
in \cite{DF}. The general case for $\PP 1 \times \PP 2$ embedded
in any bi-degree $(d_{1}, d_{2})$ is done in \cite{BD}. In
\cite{ChCi} the case $\PP 1 \times \PP n$ embedded i bi-degree $
(d,1)$ is treated.

In \cite{CaCh} one can find the defective cases $\PP 2 \times \PP
3$ embedded in bi-degree $(1,2)$, $\PP 3 \times \PP 4$ embedded in
bi-degree $(1,2)$ and $\PP 2 \times \PP 5$ embedded in bi-degree
$(1,2)$.

The paper \cite{CGG} studies also the cases $\PP n \times \PP m$
with bi-degree $ (n+1,1)$; $\PP 1 \times \PP 1$ with bi-degree $
(d_{1}, d_{2})$ and $\PP 2 \times \PP 2 $ with bi-degree $(2,2)$.
In \cite{Ab} the cases $\PP 1 \times \PP m $ in bi-degree
$(2d+1,2)$, $\PP 1 \times \PP m$ in bi-degree $ (2d,2)$, and $\PP
1 \times \PP m $ in bi-degree $ (d,3)$ can be found. A recent
result on  $\PP n \times \PP m$ in bi-degree $ (1,2)$ is in
\cite{AB}, where the authors prove the existence of two functions
$\underline{s}(n,m)$ and $\overline{s}(n,m)$ such that
$\sigma_s(X^{(n,m)}_{(1,2)})$ has the expected dimension for
$s\leq \underline{s}(n,m)$ and for $s\geq \overline{s}(n,m)$. 
In the same paper it is also shown that $X^{(1,m)}_{(1,2)}$ is never
defective and all the defective cases for $X^{(2,m)}_{(1,2)}$ are
described.

The varieties $\PP n \times \PP m$ embedded in bi-degree $(1,d)$
are related to the study of Grassmann defectivity (\cite{DF}).
More precisely, one can consider the Veronese variety $X$ obtained by embedding $\PP m$ in
$\PP N$ using the $d$-uple Veronese
embedding ($N= {m+d \choose d}$). Then consider, in $\mathbb G(n,N)$, the
{\it $(n, s-1)$-Grassmann secant variety} of $X$, that is, the closure of the
set of $n$-dimensional linear spaces contained in the linear span of $s$ linearly
independent points of $X$. The variety $X$ is said to be {\it $(n,s-1)$-Grassmann
defective} if the $(n, s-1)$-Grassmann secant variety of $X$ has not the expected dimension. It is
shown in \cite{DF}, following Terracini's ideas in \cite{Te1},
that $X$ is $(n,s-1)$-Grassmann defective if and only if
 the $s^{th}$ higher secant varieties of
the embedding of  $\PP n \times \PP m$ into $\PP N$ ($N=(n+1){m+d
\choose d}-1$ ), by the sections of the sheaf $\mathcal{O}(1,d)$,
 is $(s-1)$-defective. Hence, the result
proved in this paper gives information about the Grasmann
defectivity of Veronese varieties (see Remark \ref{grassref}).

The main result of this paper is Theorem \ref{t1} where we prove
the regularity of the Hilbert function of a subscheme of $\PP
{n+m}$ made of a $d$-uple $\PP {n-1}$, $t$ projective subspaces of
dimension $n$ containing it, a simple $\PP {m-1}$ and a number of
double points that is an integer multiple of $n-1$. This theorem,
together with Theorem 1.1 in \cite{CGG} (see Theorem \ref
{metodoaffineproiettivo} in this paper),  gives immediately the
regularity of the higher  secant varieties of the Segre-Veronese
variety that we are looking for.

More precisely, we consider (see  Section \ref{results}) the case
of $\PP n \times \PP m$ embedded in bi-degree $(1,d)$ for  $d\geq
3$. We prove (see Theorem \ref{t2})  that the $s^{th}$ higher
secant variety of such Segre-Veronese varieties  have the expected
dimensions for $s \leq s_1$ and for $s \geq s_2$, where

$$s_1= \max
 \left  \{
 s \in \mathbb N \  | \ s \  is \  a \  multiple\   of\
 (n+1) \ and  \  s  \leq
 \left\lfloor \frac{(n+1){m+d\choose d}}{m+n+1}  \right\rfloor
\right
 \},
 $$
 $$
s_2= \min
 \left  \{
 s \in \mathbb N \  | \ s \  is \  a \  multiple\   of\
 (n+1) \ and  \  s  \geq
 \left\lceil \frac{(n+1){m+d\choose d}}{m+n+1}  \right\rceil
\right
 \}.
 $$

%%%%%%%%%%%%%%%%%%%%
\section{Preliminaries and Notation}\label{prelim}
%%%%%%%%%%%%%%%%%%%%

We will always work with projective spaces defined over  an
algebraically closed field $K$ of characteristic $0$. Let us
recall the notion of higher secant varieties and some
 classical results which we will often use.
 For definitions and proofs we refer the reader to \cite{CGG} .

\begin{defi}\label{secant} {\rm Let $X\subset \PP N$ be a projective variety.
We define the
 $s^{th}$ {\it higher secant variety}  of $X$, denoted by $\sigma_{s}(X)$, as the Zariski closure of the union of all linear spaces spanned by  $s$ points of $X$, i.e.:
$$\sigma_{s}(X):= \overline{ \bigcup_{P_{1}, \ldots , P_{s}\in  X} \langle P_{1}, \ldots , P_{s} \rangle}\subset \PP N.$$

When $\sigma_s(X)$ does not have the expected dimension, that is
$\min\{N, s(\dim X+1)-1\},$
$X$ is
said to be $(s-1)$-{\it defective}, and the positive integer
$$
\delta _{s-1}(X) = {\rm min} \{N, s(\dim X+1)-1\}-\dim \sigma_s(X)
$$
is called the  $(s-1)${\it-defect} of $X$. }\end{defi} The basic
tool to compute the dimension of $\sigma_{s}(X)$ is Terracini's
Lemma (\cite{Te}):

\begin{lemma}
[{\bf Terracini's Lemma}]
 Let $X$  be an irreducible
variety in $\mathbb P^ N$, and let $P_1,\ldots ,P_s$  be s generic
points on $X$. Then, the tangent space to $\sigma_{s}(X)$ at a
generic point in $ \langle P_1,\ldots ,P_s \rangle$  is the linear span in
$\mathbb P^ N$ of the tangent spaces $T_{X, P_i}$ to $X$ at $P_i$,
$i=1,\ldots ,s$, hence
$$ \dim \sigma_{s}(X) = \dim  \langle T_{X,P_1},\ldots ,T_{X,P_s}\rangle.$$
\end{lemma}

A consequence of Terracini's Lemma is the following corollary (see
\cite[Section 1]{CGG} or  \cite [Section 2]{AB} for a proof of
it).

\begin {corol}\label{corTer}
Let $X^{(n,m)}_{(1,d)} \subset \mathbb P^N$ be the {\it
Segre-Veronese variety}  image of  the embedding of $\mathbb P^n
\times \mathbb P^m$ by the sections of the sheaf ${\cal O}(1, d)$
into $\mathbb P^N$, with $N=(n+1) {m+d \choose d}-1$. Then
$$
\dim   \sigma_s \left ( X^{(n,m)}_{(1,d)}   \right ) = N - \dim (I_Z)_{(1,d)} =
 H (Z,(1,d))  -1 ,
$$
where $Z \subset \mathbb P^n \times \mathbb P^m$ is a set of $s$
generic double points, $I_Z $ is the multihomogeneous ideal of $Z$
in $R = K [x_0, \dots, x_n,y_0, \dots, y_m]$, the multigraded
coordinate ring of $ \mathbb P^n\times \mathbb P^m$, and  $ H
(Z,(1,d)) $ is the multigraded Hilbert function of $Z$.
\end{corol}

Now we recall  the fundamental  tool which allows us to convert
certain questions about ideals of varieties in multiprojective
space to questions about ideals in standard polynomial rings (for
a more general statement see \cite[Theorem 1.1]  {CGG}) .

\begin{theorem}\label{metodoaffineproiettivo}
Let $X^{(n,m)}_{(1,d)} \subset \mathbb P^N$  and $Z \subset
\mathbb P^n \times \mathbb P^m$ as in Corollary \ref{corTer}. Let
$H_{1}, H_{2}\subset \PP {n+m}$ be  generic projective linear
spaces of dimensions $n-1$ and $m-1$, respectively, and let
$P_{1}, \dots ,P_{s} \in \mathbb P^{n+m}$  be
generic points. Denote by
$$dH_1+H_2+2P_{1}+ \cdots + 2P_s \subset \mathbb P^{n+m}$$
 the scheme defined
by the ideal sheaf
${\mathcal I}^{d}_{H_1}\cap {\mathcal I}_{H_2} \cap {\mathcal I}^{2}_{P_1}\cap \dots \cap
{\mathcal I}^{2}_{P_s}\ \subset {\mathcal O}_{P^{n+m }  }$.
Then
$$
 \dim (I_Z)_{(1,d)} = \dim (I_{dH_1+H_2+2P_1+ \cdots + 2P_s})_{d+1}
$$
hence
$$
\dim  \sigma_s \left ( X^{(n,m)}_{(1,d)}   \right )  = N - \dim (I_{dH_1+H_2+2P_1+ \cdots + 2P_s})_{d+1}  .
$$
\end{theorem}

Since we will make use of Castelnuovo's inequality several times,
we recall it here  (for notation and proof we refer to \cite{AH2},
Section 2).

 \begin{lemma} [{\bf Castelnuovo's inequality}] \label{castelnuovo}
Let $H \subset \mathbb P^N$ be a hyperplane, and let $X \subset \mathbb P
^N$ be a  scheme. We denote by $Res_H X$
the scheme defined by the ideal $(I_X:I_H)$ and we call it the
{\it residual scheme} of $X$ with respect to $H$, while the scheme
$Tr_H X \subset H$ is the schematic intersection $X\cap H$, called
the {\it trace} of $X$ on $H$.
Then
$$
\dim (I_{X, \mathbb P^N})_t  \leq  \dim (I_{ Res_H X, \mathbb P^N})_{t-1}+
\dim (I_{Tr _{H} X, H})_t.
$$
\end{lemma}

%%%%%%%%%%%%%%%%%%%%
\section{Segre-Veronese embeddings of $\mathbb P^n \times \mathbb P^m$ }\label{results}
%%%%%%%%%%%%%%%%%%%%

Now that we have introduced all the necessary tools that we need for the main theorem of this paper we can state and prove it.

\begin{theorem}\label{t1}
 Let   $d\geq 3$, $n,m \geq 1$ and  let $s=(n+1)q$ be an integer
multiple of $n+1$. Let  $ P_1, \dots ,P_s \in \mathbb P^ {n+m}$ be
generic points and $H_1\simeq \mathbb P^{n-1}, H_2\simeq \mathbb
P^{m-1}$ be generic linear spaces  in $\mathbb P^ {n+m}$.
 Let $W_{1}, \ldots , W_{t}\subset \mathbb P^ {n+m}$ be $t$ generic linear spaces of dimension $n$ containing $H_{1}$.
 Now consider the scheme
\begin{equation}\label{X}
\mathbb{X}:=dH_{1}+H_{2}+2P_{1}+ \cdots + 2P_{s}+W_{1}+ \cdots + W_{t}
\end{equation}
Then for any $q, t\in \mathbb{N}$ the dimension of the degree
$d+1$ piece of the ideal $I_{\mathbb{X}}$ is the expected one,
that is
$$\dim(I_{\mathbb{X}})_{d+1}=\max \left \{ (n+1){m+d \choose d}-s(n+m+1)-t(n+1)\ ; 0\right \}. $$

\end{theorem}

\begin{proof}
We will prove the theorem by induction on $n$.

A  hypersurface   defined by a form of $(I_{dH_1})_{d+1}$ cuts on
$W_i \simeq \mathbb P^n$ a   hypersurface which has $H_1$ as a
fixed component of multiplicity $d$.
 It follows that
$$\dim (I_{dH_1 ,W_i}  )_{d+1} = \dim (I_{\emptyset,W_i}  )_1=n+1.
$$
Hence the expected number of conditions that a linear space $W_i$ imposes  to the forms of
$(I_{\mathbb{X}})_{d+1}$ is at most $n+1$.
Moreover a double point imposes at most $n+m+1$ conditions.
So, since by Theorem \ref{metodoaffineproiettivo} with $Z = \emptyset $ we get
$$
\dim (I_{dH_1+H_2})_{d+1}  =\dim  R_{(1,d)} = (n+1){m+d \choose d},
$$
(where $R = K [x_0, \dots, x_n,y_0, \dots, y_m] $), then we have
\begin{equation}\label{disug}
 \dim(I_{\mathbb{X}})_{d+1} \geq    (n+1){m+d \choose d}-s(n+m+1)-t(n+1).
 \end{equation}
\\

Now let $H\subset \PP {n+m}$ be a generic hyperplane containing
$H_{2}$ and let $\widetilde{ \mathbb X}$ be the scheme obtained
from $ \mathbb  X$ by  specializing the $nq$ points
$P_1,\dots,P_{nq}$ on $H$, ($P_{nq+1}, \dots , P_{s}$ remain
generic points, not lying on $ H$).

Since by the semicontinuity of the Hilbert Function $ \dim(I_{
\widetilde {  \mathbb X } })_{d+1} \geq
  \dim(I_{  {  \mathbb X } })_{d+1}$, by (\ref{disug}) we have
\begin{equation}\label{disug2}
 \dim(I_{  \widetilde {  \mathbb X } })_{d+1} \geq
 (n+1){m+d \choose d}-s(n+m+1)-t(n+1).
 \end{equation}

Let $V_i = \langle H_1, P_i \rangle \simeq \mathbb P^n $. Since the linear spaces $V_i  $
   are in the base locus of the hypersurfaces defined by the forms of
   $(I_{  \widetilde {  \mathbb X } })_{d+1}$, we have
\begin{equation}\label{spazifissi}
   (I_{  \widetilde {  \mathbb X } })_{d+1}=
   (I_{  \widetilde {  \mathbb X }+ V_1+\dots+V_{s} })_{d+1} .
\end{equation}

Consider the residual scheme of  $   ( \widetilde {  \mathbb X }    + V_1+\dots+V_s  ) $ with
respect to $H$:
 $$
 Res_{H}   ( {  \widetilde {  \mathbb X }   }   + V_1+\dots+V_s )=dH_{1}+W_{1}+ \cdots + W_{t}+P_{1}+\cdots + P_{nq}+2P_{nq+1}+ \cdots + 2P_{s}  + V_1+\dots+V_{s}
 $$
 $$=dH_{1}+W_{1}+ \cdots + W_{t}+2P_{nq+1}+ \cdots + 2P_{s}  + V_1+\dots+V_{s}
 \subset \mathbb P^{n+m}   .
 $$
Any form of degree $d$ in  $I_{ Res _H  { ( \widetilde {  \mathbb X }   }  + V_1+\dots+V_s)} $
represents a cone whose vertex contains
$H_1$.
Hence if $   \mathbb Y   \subset \mathbb P^m$ is the scheme obtained by projecting  $ Res _H  ({  \widetilde {  \mathbb X }   }+  V_1+\dots+V_s )$
 from $H_1$ in a $\mathbb P^m$, we have:
\begin{equation} \label{res1}
\dim(I_{ Res _H  { ( \widetilde {  \mathbb X }   }  + V_1+\dots+V_s)})_d = \dim(I_{   \mathbb Y    } )_d .
\end{equation}

 Since the image by this projection of each $W_i$ is a point, and  for $1 \leq i \leq nq$ the image of $P_i+V_i$ is a simple point, and
 for $nq+1 \leq i \leq s$ the image of $2P_i+V_i$ is a double point,
 we have that  $   \mathbb Y$ is a  scheme
 consisting of  $t+ nq $ generic  points  and $ q $ generic  double points.

Now by the Alexander-Hirschowitz Theorem (see \cite{AH}), since
$d>2$ and $t+nq >1$  we have that the dimension of the degree $d$
part of the ideal of  $q$ double points plus $t+ nq  $ simple
points is always as expected. So we get
\begin{equation} \label{res2}
 \dim(I_{   \mathbb Y    } )_d =\max \left \{  {m+d\choose d} - q(m+1)-t-nq  \ ;\  0 \right
 \}.
\end{equation}
Now let $n=1$. In this case we have: $s=2q$,
\begin{equation}\label{resxn=1}
\dim(I_{ Res _H  { ( \widetilde {  \mathbb X }   }  + V_1+\dots+V_s)})_d =
 \dim(I_{   \mathbb Y    } )_d =\max \left \{  {m+d\choose d} - q(m+1)-t-q  \ ;\  0 \right \},
\end{equation}
moreover $H_1$ is a point, $H_{1}\cap H$ is the empty set, the
$W_i$ and the $V_i$ are lines, and $V_i$ is the line $H_1P_i$.

Set  $W'_{i}=W_{i}\cap H$, $V'_{i}=V_{i}\cap H$.  Note that for $1\leq i \leq q$ we have $V'_{i}=P_i.$
The trace on $H$ of
$     {   \widetilde {  \mathbb X }  +V_1+ \dots +V_{s} } $ is:
$$Tr_{H}    (   {   \widetilde {  \mathbb X }   } +V_1+ \dots +V_{s})
=H_2+2P_{1}+ \cdots + 2P_{q}+ W'_{1}+ \cdots + W'_{t}+V'_1+ \dots +V'_{2q} =
$$
$$=H_2+2P_{1}+ \cdots + 2P_{q}+ W'_{1}+ \cdots + W'_{t}+V'_{q+1}+ \dots +V'_{2q}
  \subset H \simeq  \PP {m}
  .$$
  So $Tr_{H}    (   {   \widetilde {  \mathbb X }   } +V_1+ \dots +V_{s})\subset H$  is a scheme in $\PP {m}$ union of $H_2 \simeq \PP {m-1}$, plus $q$ generic double points and $t+q$ generic simple points.
  As above, by \cite{AH}, since $d>2$ and $t+q \geq1$  we get
 $$
\dim(I_{ Tr _H  { ( \widetilde {  \mathbb X }   }  + V_1+\dots+V_s)})_{d+1} =
\dim(I_{  2P_{1}+ \cdots + 2P_{q}+ W'_{1}+ \cdots + W'_{t}+V'_{q+1}+ \dots +V'_{2q}  })_{d}
$$
\begin{equation}\label{trxn=1}
=\max \left \{  {m+d\choose d} - q(m+1)-t-q  \ ;\  0 \right \}.
\end{equation}

By Castelnuovo's inequality (see Lemma \ref {castelnuovo}), by (\ref{resxn=1}) and  (\ref{trxn=1})
we get
\begin{equation}\label{disug3}
\dim(I_{{ \widetilde {  \mathbb X }   }  + V_1+\dots+V_s})_{d+1} \leq
\max \left \{ 2 {m+d\choose d} - 2q(m+1)-2t-2q  \ ;\  0 \right \},
\end{equation}
so by   (\ref{disug2}),  (\ref{spazifissi})  and (\ref{disug3}) we have
 $$
 \dim(I_{{ \widetilde {  \mathbb X }   } })_{d+1} =
 \max \left \{ 2 {m+d\choose d} - 2q(m+2)-2t  \ ;\  0 \right \}.
$$
From here, by   (\ref{disug}) and by the semicontinuity of the Hilbert Function we get
 $$
 \dim(I_{{  \mathbb X }   })_{d+1} =
 \max \left \{ 2 {m+d\choose d} - 2q(m+2)-2t  \ ;\  0 \right \}
$$
and the result is proved  for $n=1$.

Let $n>1$.

Set: $H'_{1}=H_{1}\cap H$;  $W'_{i}=W_{i}\cap H$;  $V'_{i}=V_{i}\cap H$.
With this notation the trace of
$     {   \widetilde {  \mathbb X }  +V_1+ \dots +V_s } $ on $H$ is:
$$Tr_{H}    (   {   \widetilde {  \mathbb X }   } +V_1+ \dots +V_s)
=dH'_{1}+H_2+2P_{1}+ \cdots + 2P_{nq}+ W'_{1}+ \cdots + W'_{t}+V'_1+ \dots +V'_s
  \subset H \simeq  \PP {n+m-1}
  .$$

Anaugously as above, observe that the linear spaces $V'_i = \langle H'_1, P_i \rangle \simeq \mathbb P^n $
   are in the base locus for the hypersurfaces defined by the forms of
   $(I_{ dH'_{1}+2P_{i} })_{d+1}$, hence the parts of degree $d+1$ of the ideals of
   $Tr_{H}    (   {   \widetilde {  \mathbb X }   } +V_1+ \dots +V_s)$ and
   of
     $Tr_{H}    (   {   \widetilde {  \mathbb X }   } +V_{nq+1}+ \dots +V_s)$
    are equal. So we have
   $$
 (I_{ Tr_{H}    (   {   \widetilde {  \mathbb X }   } +V_1+ \dots +V_s)})_{d+1}=
(I_ { Tr_{H}    (   {   \widetilde {  \mathbb X }   } +V_{nq+1}+ \dots +V_s)} )_{d+1}=
(I_{         \mathbb T      })_{d+1} ,
   $$
   where
   $$\mathbb T=
   dH'_{1}+H_2+2P_{1}+ \cdots + 2P_{nq}+ W'_{1}+ \cdots + W'_{t}+V'_{nq+1}+ \dots +V'_s
   \subset   \PP {n+m-1},
   $$
   that is, $\mathbb T$
is union of the $ d$-uple linear space $H_1'  \simeq \mathbb P^{n-2}$, the linear space
$H_2\simeq \mathbb P^{m-1}$,  $t+q$ generic linear spaces through $H_1'  $,  and
 $nq$ double points. Hence by the inductive hypothesis we have
\begin{equation}\label{tr1}
\dim(I_{\mathbb{T}})_{d+1}=\max \left \{ n{m+d \choose d}-nq(n+m)-(t+q)n\ ; 0\right \}.
\end{equation}

By (\ref{spazifissi}),   by Lemma \ref{castelnuovo}, by (\ref{res1}), (\ref{res2}) and (\ref{tr1})
we get
$$
  \dim (I_{  \widetilde {  \mathbb X } })_{d+1}\leq
  \max \left \{  {m+d\choose d} - q(m+1)-t-nq  \ ;\  0 \right \} +
  \max \left \{ n{m+d \choose d}-nq(n+m)-(t+q)n\ ; 0\right \}
  $$
  $$=
    \max \left \{  {m+d\choose d} - q(m+1)-t-nq  \ ;\  0 \right \} +
 \max \left \{ n \left ( {m+d\choose d} - q(m+1)-t-nq \right )\ ; 0\right \}
  $$
  $$=
   \max \left \{ (n+1) \left ( {m+d\choose d} - q(m+1)-t-nq \right )\ ; 0\right \}
  $$
  $$=
  \max \left \{ (n+1){m+d \choose d}-s(n+m+1)-t(n+1)\ ; 0\right \}
  .$$
Now the conclusion follows from   (\ref{disug}) and the semicontinuity of the Hilbert Function
 and this ends the proof.
\end{proof}

%%%%%%%%%%%%%%%%%%%%%%%%%%%%%%%
%%%%%%%%%%%%%%%%%%%%%%%

\begin{corol}\label{c1}
 Let   $d\geq 3$, $n,m \geq 1$ and  let
 $$s_1:= \max
 \left  \{
 s \in \mathbb N \  | \ s \  is \  a \  multiple\   of\
 (n+1) \ and  \  s  \leq
 \left\lfloor \frac{(n+1){m+d\choose d}}{m+n+1}  \right\rfloor
\right
 \}
 $$
 $$
s_2:= \min
 \left  \{
 s \in \mathbb N \  | \ s \  is \  a \  multiple\   of\
 (n+1) \ and  \  s  \geq
 \left\lceil \frac{(n+1){m+d\choose d}}{m+n+1}  \right\rceil
\right
 \}.
 $$
 Let  $ P_1, \dots ,P_{s}  \in \mathbb P^ {n+m}$  be generic points and $H_1\simeq \mathbb P^{n-1}, H_2\simeq \mathbb P^{m-1}$ be generic linear spaces  in $\mathbb P^ {n+m}$.
 Consider the scheme
$$
\mathbb{X}:=dH_{1}+H_{2}+2P_{1}+ \cdots + 2P_{s}.
$$
Then for any $s \leq s_1$ and any $s \geq s_2$ the dimension of
$(I_{\mathbb{X}})_{d+1}$ is the expected one, that is
$$\dim(I_{\mathbb{X}})_{d+1}=
\left \{
\begin{matrix} (n+1){m+d \choose d}-s(n+m+1) &\   for  \ s  \leq s_1   \\
\\
0 & for  \ s  \geq s_2   \\

\end{matrix}
\right.
$$

\end{corol}

\begin{proof}
By applying Theorem  \ref{t1}, with $t=0$, to the scheme
$
\mathbb{X}=dH_{1}+H_{2}+2P_{1}+ \cdots + 2P_{s}
$, we get that the dimension of
$I(\mathbb{X})_{d+1}$ is the expected one for $s=(n+1)q$ and for
any $q\in \mathbb{N}$. Hence if  $s_1$ is the biggest integer
multiple of $n+1$ such that $\dim(I_{\mathbb{X}})_{d+1}\neq 0$ we
get that for that value of $s$ the Hilbert function
$H(I_{\mathbb{X}},d+1)$ has the expected value. Now if for such
$s_1$ we have that $(I_{\mathbb{X}})_{d+1}$ has the expected
dimension than it has the expected dimension also for every $s\leq
s_1$.

Now, if $s_2$ is the smallest integer multiple of $n+1$ such that $\dim(I_{\mathbb{X}})_{d+1} = 0$  then obviously such a dimension will be zero for all $s\geq s_2$.
\end{proof}

%%%%%%%%%%%%%%%%%%%%%%%%%%%%%%%%%%%%

\begin{theorem}\label{t2}

 Let   $d\geq 3$, $n,m \geq 1$, $N  = (n+1){m+d \choose d}-1$ and  let $s_1, s_2$ be as in Corollary \ref{c1}.

Then the variety $\sigma_{s} \left ( X^{(n,m)}_{(1,d)}\right ) \subset \mathbb P^N$
 has the expected dimension for any
 $s \leq s_1$ and any $s \geq s_2$,
 that is
$$
\dim
\sigma_{s} \left ( X^{(n,m)}_{(1,d)}
\right )=
\left \{
\begin{matrix}
s(n+m+1) -1 &\   for  \ s  \leq s_1   \\
\\
 N & for  \ s  \geq s_2   \\

\end{matrix}
\right. .
$$

\end{theorem}

\begin{proof}
 Let $H_{1}, H_{2}\subset \PP {m+n}$ be projective subspaces of dimensions $n-1$ and $m-1$ respectively and let $P_{1}, \ldots , P_{s}\in \PP {n+m}$ be $s$ generic points of $\PP {n+m}$. Define $\mathbb{X}\subset \PP {m+n}$ to be the scheme $\mathbb{X}:=dH_{1}+H_{2}+2P_{1}+ \cdots + 2P_{s}$. Theorem 1.1 in \cite{CGG} shows that $\dim \sigma_{s} \left ( X^{(n,m)}_{(1,d)}
\right )$ is the expected one if and only if $\dim(I_{\mathbb{X}})_{d+1}$ is the expected one.
 Therefore the conclusion immediately follows from Theorem \ref{metodoaffineproiettivo} and Corollary \ref{c1}.
\end{proof}

\begin{rem}{\em If  ${m+d \choose d}$ is multiple of $(m+n+1)$, say
${m+d \choose d} =  h(m+n+1)$, we get

$$
 \left\lfloor \frac{(n+1){m+d\choose d}}{m+n+1}  \right\rfloor =
  \left\lceil \frac{(n+1){m+d\choose d}}{m+n+1}  \right\rceil =
h(n+1)
$$
so  $s_1 = s_2$. Hence in this case   the variety
$\sigma_{s} \left ( X^{(n,m)}_{(1,d)}\right )$ has the expected dimension for any $s$.

If ${m+d \choose d}$ is not multiple of $(m+n+1)$, it is easy to show that
$s_2- s_1 = n$.  Thus   there are at most $n$ values of $s$ for which
the $s^{th}$ higher secant varieties of  $X^{(n,m)}_{(1,d)}$  can be defective.
}
\end{rem}

\begin{rem}\label{grassref}{\em Theorem \ref{t2} has a straightforward interpretation in terms of Grassmann defectivity. More precisely, we see that the $d$-uple Veronese embedding of $\PP m$ is not $(n,s-1)$-Grassmann defective when $s\leq s_1$ or $s\geq s_2$.
}
\end{rem}

\end{document}